\documentclass[11pt,double,reqno]{amsart}

\usepackage{amsmath, amsthm, amscd, amssymb, amsfonts, amsbsy}

\usepackage{tikz}
\usepackage{times}
\usepackage{graphicx}
\usepackage{epsfig}
\usepackage{color}

\usepackage{enumerate}

\usepackage[bitstream-charter]{mathdesign}
\usepackage[T1]{fontenc}

\newtheorem{theo}{Theorem}[section]
\newtheorem{defin}[theo]{Definition}
\newtheorem{prop}[theo]{Proposition}

\newtheorem{lemm}[theo]{Lemma}
\newtheorem{rem}[theo]{Remark}

\numberwithin{equation}{section}

\newcommand{\set}[1]{{\left\lbrace #1 \right\rbrace}}
\newcommand{\norm}[1]{{\left\Vert #1 \right\Vert}}

\newcommand{\bD}{\mathbf{D}}

\newcommand{\bF}{\mathbb{F}}
\newcommand{\bI}{\mathbb{I}}
\newcommand{\Z}{\mathbb{Z}}

\newcommand{\cS}{{\mathcal S}}

\newcommand{\al}{\alpha}
\newcommand{\be}{\beta}
\newcommand{\ga}{\gamma}
\newcommand{\Ga}{\Gamma}

\newcommand{\om}{\omega}
\newcommand{\Om}{\Omega}

\newcommand{\si}{\sigma}

\newcommand{\ep}{\epsilon }
\newcommand{\te}{\theta}
\newcommand{\De}{\Delta}
\newcommand{\de}{\delta}

\newcommand{\R}{{\mathbb R}^n}

\newcommand{\ri}{\rightarrow}
\newcommand{\Rn}{{\mathbb R}^{n}}

\newcommand{\na}{\nabla}

\newcommand{\intRn}{\int_{{\mathbb R}^n}}

\DeclareMathOperator*{\divg}{div\,}

\begin{document}
\baselineskip=18pt

\title[]{Global existence of solutions of the stochastic incompressible non-Newtonian fluid models}

\
\author{Tongkeun Chang}

\author{Minsuk Yang}

\address{Department of Mathematics, Yonsei University \\
Seoul, 136-701, South Korea}

\email{C.: chang7357@yonsei.ac.kr, Y.: m.yang@yonsei.ac.kr}

\begin{abstract}
In this paper, we study the existence of solutions of stochastic incompressible non-Newtonian  fluid models in $\R$.  For the  existence of solutions, we assume that the extra stress tensor $S$ is represented by $S({\mathbb A}) = {\mathbb F} ( {\mathbb A}) {\mathbb A}$ for $ n \times n$ matrix ${\mathbb G}$. We assume that ${\mathbb F}(0) $ is uniformly elliptic matrix and
\begin{align*}
|{\mathbb F}({\mathbb G})|, \,\, 
 | D {\mathbb F} ({\mathbb G})|, \,\, | D^2({\mathbb F} ({\mathbb G})  ){\mathbb G}| \leq c \quad \mbox{for all} \quad 0 <  |{\mathbb G}| \leq r_0
 \end{align*}
 for some $r_0 > 0$.
Note that ${\mathbb F}_1$ and  ${\mathbb F}_2$ for  $ d \in {\mathbb  R}$, and   ${\mathbb F}_3$ for $d \geq 3$  introduced in \eqref{0207-1} satisfy our assumption.\\


\end{abstract}

\maketitle

\section{Introduction}
\setcounter{equation}{0}

In this paper, we study  the existence of solutions to stochastic incompressible non-Newtonian fluid models governed by the system:
\begin{equation}
\label{maineq2}
\begin{split}
d u &= ( \divg ( \bF (\bD u) \bD u - u\otimes u ) - \na p ) dt + g dB_t, \quad 
\divg u = 0, \quad u|_{t =0} = u_0
\end{split}
\end{equation}
for $(t, x) \in (0, \infty) \times \R$, $ n \geq 2$ with  the given initial condition   $u_0$   and   random noise  $g(t, x)dB_t$, where $B_t$ denotes Brownian motion.
Here, the diffusion term $\divg ( \bF (\bD u) \bD u)$ models the non-Newtonian stress tensor. 
The symmetric gradient $\bD u$ is defined as
\[
\bD u = \frac12 (\na u + (\na u)^T),
\]
where $(\na u)^T$ is the transpose of the gradient matrix $\na u$.
When $\bF(\bD u) = \bI$ ( the $n \times n$ identity matrix), the system  \eqref{maineq2} reduces to the stochastic Navier-Stokes equations, which describe standard Newtonian fluids.

Common choices for the stress tensor $\bF(\bD u)$ include:
\begin{equation}
\label{0207-1}
\begin{split}
\bF_1(\bD u)& = (\mu_0 + \mu_1|\bD u|)^{d -2} \bI, \\
\bF_2(\bD u)& = (\mu_0 + \mu_1 |\bD u|^2)^{\frac{d-2}2} \bI, \\
\bF_3(\bD u) &= (\mu_0 + \mu_1|\bD u|^{d-2} ) \bI,
\end{split}
\end{equation}
where $\mu_0$ and $\mu_1$ are positive constants.
The norm of a matrix ${\mathbb A} = (a_{ij})_{1 \leq ij \leq n}$ is given by $|{\mathbb A}| = \big( \sum_{1\leq i,j \leq n} |a_{ij}|^2 \big)^\frac12$.
The parameter $d$ characterizes the fluid's behavior. 
Shear-thinning fluids correspond to $1 < d < 2$ and shear-thickening fluids correspond to $2 < d < \infty$.
 
We review prior results on the existence of solutions relevant to our study. 
The deterministic incompressible non-Newtonian fluid equations have been extensively studied by many mathematicians.
Ladyžhenskaja (\cite{La0,La}) was the first to establish the existence of weak solutions for \eqref{maineq2} in bounded domains under the no-slip boundary condition $( u|_{x_n} =0)$.
Many prominent mathematicians, including Nečas, have made significant contributions to this topic (see \cite{Am,  BW,  BD, FMS, KKK, MNR, V, V2, VKR, Wol}).

Pokorný \cite{Pok} studied deterministic incompressible non-Newtonian fluids for the type ${\mathbb F}_3$ in $\R$. 
He established the local existence of weak solutions to \eqref{maineq2} for $d > 1$ and proved the uniqueness of weak solutions for $d > 1 + \frac{2n}{n+2}$ when $n \leq 3$. 
Kang et al. \cite{KKK} demonstrated both the local and global existence of regular solutions to \eqref{maineq2} for type ${\mathbb F}_2$. 
Bae and Kang \cite{BK} further established the existence of weak solutions to \eqref{maineq2} for the  type  ${\mathbb F}_2$ under specific conditions. 
More recently, Chang and Jin \cite{CJ-jmp} proved the global existence of weak solutions to \eqref{maineq2} for the  types ${\mathbb F}_1$ and ${\mathbb F}_2$ for all $ d \in {\mathbb R}$, and for type ${\mathbb F}_3$ when $ d \geq 3$. 

For stochastic partial differential equations in $\R$, Du and Zhang \cite{MR4002154} established the existence of solutions to the stochastic Navier–Stokes equations in Chemin–Lerner-type Besov spaces.
More recently, Chang and Yang \cite{CY}  proved the existence of solutions to the stochastic Navier–Stokes equations in weighted $L^p$-spaces.

Motivated by these results, we aim to establish the existence of solutions to stochastic incompressible non-Newtonian fluid models in $\R$.
We emphasize that our solutions are constructed in standard Besov spaces.

To proceed, we assume that $\bF(\bD u)$ satisfies the following regularity and structural properties:
{\bf Assumption}
\begin{enumerate}\label{assumption}
\item
$|\bF({\mathbb G})|, \,\, |D \bF({\mathbb G})|,\,\, |D^2 \big( \bF({\mathbb G})  \big) {\mathbb G} | \leq c $ for $ 0 < |{\mathbb G}| \leq r_0$ for some $ 0<r_0$.
\item
$\bF(0)$ is uniformly elliptic, meaning that there exist positive constants $c_1$ and $c_2$ such that $c_1|\xi|^2 \leq  \xi_i\bF(0)_{ij} \xi_j \leq c_2 |\xi|^2$ for all $ \xi \in \R$.
\end{enumerate}

It is worth noting that the functions
$\bF_1$ and $\bF_2$ for   $d \in {\mathbb R}$, as well as $\bF_3$ for $3 \leq d$ in \eqref{0207-1}, satisfy these assumptions. 
This work extends existing results by relaxing regularity constraints while maintaining physical relevance in modeling non-Newtonian fluids subjected to stochastic forces.

Before presenting our main result, we introduce several notations for function spaces to express various mathematical formulas concisely.
Let 
\begin{align}\label{def1}
{\mathbb T} : = \dot {\mathbb B}^{\al-\frac{2}{p} -2a}_{p,q}({\mathbb
R}^{n}) \, \cap \,\dot {\mathbb B}^{ -2a}_{p,q}(\Rn) \, \cap \, \dot {\mathbb W}_\infty^1 (\R)
\end{align}
and 
\begin{align}\label{def2}
\begin{split}
{\mathcal T} : 
&= {\mathcal L}_{a}^{q}( \Om\times (0,\infty), {\mathcal P}; \dot B^{\al-1}_{p,q} (\R)) \cap {\mathcal L}^{q}_{a_2} (\Om \times (0, \infty); {\mathcal P}, L^{p_2}(\R)) \\
&\quad \cap {\mathcal L}^{q} (\Om \times (0, \infty); {\mathcal P}, \dot H^1_p(\R)) \cap {\mathcal L}^{q} (\Om \times (0, \infty); {\mathcal P}, \dot H^2_p(\R)),
\end{split}
\end{align}
where  $ 1 < \al < 2$,   $n< p< \infty$, $ 2 \leq  q < \infty$ such that   
\begin{align}\label{0721-2}
\frac12 - \frac{n}{2p} -\frac1q  =a \geq 0, \quad p_2 \leq p, \quad 0 \leq a \leq  a_2, \quad \big(\frac{n}{p_2}-\frac{n}{p}\big) + 2(a_2-a) = 1.
\end{align}

For the solution spaces, we consider the following function spaces:
\begin{align*}
{\mathcal X} (T) := {\mathcal L}^q_a (0, T;  \dot B^{\al}_{p,q}  (\R)) \cap {\mathcal L}^\infty (0, T; \dot W^1_\infty (\R)) \cap {\mathcal L}^q_a (0, T;  L^p (\R)), \quad  0 < T \leq \infty.
\end{align*}

The precise definitions of these function spaces are provided in Section 2.
 
\begin{defin}[Local solution]
\label{D21}
Let $n \ge 2$ and $2 \leq p, q \leq \infty$.
Assume that $u_0$ is a ${\mathbb T} $-valued ${\mathcal F}_0$-measurable random variable, and let $g \in {\mathcal T}$.
We say that $(u, \tau)$ is a local solution to the  equations \eqref{maineq2} if $u \in {\mathcal X} (\infty)$ is a progressively measurable process and $\tau$ satisfies the stopping time condition
\[
\tau(\om) = \inf \set{0 \le T \le \infty : \norm{u(\om)}_{{\mathcal X} (T)} \ge 2R}, 
\]
where $R $ is a positive number.
Moreover, for almost surely, the function $u$ belongs to 
\[
{\mathcal L}^q_a (0, \tau|];  \dot B^{\al}_{p,q}  (\R)) \cap {\mathcal L}^\infty (0, \tau|]; \dot W^1_\infty (\R)) \cap {\mathcal L}^q_a (0, \tau|];  L^p (\R))
\]
and satisfies the weak formulation
\begin{align*}
&\int_{\R} u(t,x) \cdot \Phi(x) dx
- \int_{\R} u_0(x) \cdot \Phi(x) dx \\
&= \int^t_0 \int_{\R} \Big( \bF (\bD u) \bD u  + u \otimes u \Big) :\nabla \Phi  (x) dxdt
+ \int_0^t \int_{\R}  g  \Phi dx dB_t
\end{align*}
for all $ \Phi \in C_c^\infty(\R)$ with ${\rm div} \, \Phi =0$ and for all $t < \tau(\om)$.  
\end{defin}

Here is our main result.

\begin{theo}
\label{T1}
Let $\bF$ satisfy {\bf Assumption}.  Let $ 0 <\ep$. There is $\de_0 > 0$ such that if 
\begin{align}\label{0410-1}
u_0 & \in {\mathbb T}, \qquad 
g \in  {\mathcal T}
\end{align} 
with ${\rm div} \, u_0 =0$ and   ${\rm div} \,\, g =0$ satisfying  
\begin{align*}
& \| u_0\|_{{\mathbb T } } + \| g \|_{{\mathcal T}} 
 < \de,
\end{align*}
then the initial value problem \eqref{maineq2} has a unique local weak solution $(u,\tau)$ in the sense of Definition \ref{D21} with ${\mathbb P}(\tau =\infty) \ge 1 -\ep$.
\end{theo}

The paper is organized as follows.
Section \ref{notation} introduces the definitions of function spaces and presents several inequalities used in the proof of the main theorem.
Section \ref{Proof1} provides key estimates for the solution of the stochastic Stokes equations in $\R$. 
Section \ref{nonlinear} is dedicated to the proof of Theorem \ref{T1}.
The Appendices contain the proofs of technical lemmas for the reader’s convenience. 
While these results may be seen as variants found in different literature, we include them here for completeness.

\section{ Function spaces}
\label{notation}
\setcounter{equation}{0}
 
\subsection{Function spaces}

For $1 \le p \le \infty$, we denote the Lebesgue spaces by $L^p(\R)$ with the norm $ 
 \norm{f}_{L^p(\R)} = \left(\int_{\R} |f(x)|^p dx\right)^{1/p}$.  The Fourier transform of a Schwartz function $f \in \cS(\R)$ is given by $
\widehat{f}(\xi) = \int_{\R} e^{-2\pi i x \cdot \xi} f(x) dx$.  
The definition of the Fourier transform is extended to tempered distributions $f \in \cS'(\R)$ by duality.

Let $s\in {\mathbb R}$ and $1 < p <  \infty$. We define a
  homogeneous Sobolev space $\dot W^{s }_p ({\mathbb R}^{n})$, $ n \geq 1$ by
\begin{align}\label{0219-1}
\dot W^{s }_p  ({\mathbb R}^{n}) = \{ f \in {\mathcal S}'({\mathbb
R}^{n}) \, | \, f = I_s g :  = h_{s} * g, \quad \mbox{for some} \quad g   \in L^p ({\mathbb R}^{n}) \}
\end{align}
with norm $\|f\|_{\dot W^{s }_p  ({\mathbb R}^{n})} : = \| g \|_{L^p({\mathbb R}^{n})},$
 where
$*$ is a convolution in ${\mathbb R}^{n}$. 
Here, $h_s $  is a distribution in ${\mathbb R}^{n}$ whose Fourier transform
in ${\mathbb R}^{n}$ is defined by
\begin{eqnarray*}
\widehat  h_{s} (\xi) = c_s  |\xi|^{-s}, \quad \xi \in {\mathbb R}^n.
\end{eqnarray*}
We also define $ \dot W^1_\infty(\R)$ by 
\begin{align*}
\dot W^1_\infty(\R) : = \{ f \in {\mathcal T}' (\R) \, | \, \| Df \|_{L^\infty (\R)} < \infty \, \}.
\end{align*}

To define the Littlewood--Paley operators, we fix a function $\Psi \in \cS(\R)$ whose Fourier transform is nonnegative and satisfies
\[
\chi_\set{\frac12 \le |\xi| \le 2}(\xi)  \le \widehat \Psi(\xi) \le \chi_\set{\frac14  \le |\xi| \le 4 }(\xi) 
\]
so that for $\xi \neq 0$,
\[
\sum_{j \in \Z} \widehat\Psi(2^{-j}\xi) = 1.
\]
For $j \in \Z$, we define the associated Littlewood-Paley operator by
\[
\Delta_jf(x)
= \Psi_{2^{-j}} * f(x) = \intRn e^{2\pi ix\cdot\xi} \widehat \Psi(2^{-j}\xi) \widehat f(\xi)\;d\xi.
\]

We define the homogeneous Besov space to be the set ${\dot B^s_{p,q}}(\R)$ of all $f \in \cS'(\R)$ with the finite norm
\begin{align*}
\norm{f}_{\dot B^s_{p,q}  (\R) }
&= \big(\sum_{j \in \Z} 2^{jsq} \norm{\De_j f}_p^q\big)^\frac1q,\quad
 1 \leq q < \infty,\\
 \norm{f}_{\dot B^s_{p,\infty}  (\R)  }
&=  \sup_{j \in \Z} 2^{jsq} \norm{\De_j f}_p.
\end{align*}

\ Let $ 1 < p, \, q < \infty$ and $ 0 < s < 1$. The $\dot B_{p,q}^{s } ({\mathbb R}^{n})$-norm of $f$ is equivalent to 
\begin{align}\label{0307-1}
\| f\|_{\dot B^{s }_{p,q} (\R) } & =   \Big(   \int_{{\mathbb R}^{n}}  
\frac1{|y|^{n  + qs  } } \big( \int_{{\mathbb R}^{n}}  |f(x +y ) - f (x )|^p  dx  \big)^{\frac{q}p}  dy   \Big)^\frac1q.
\end{align}
 
See Theorem 3.6.1 in \cite{amman-anisotropic} and Section 6.8 in \cite{BL}.

For $  0 < s$,   the inhomogegneous Besov space $ B^s_{p,q}(\R)$  is defined by $ B^s_{p,q}(\R) = L^p (\R) \cap \dot B^s_{p,q}(\R)$.

\begin{rem}\label{rem0712}
\begin{itemize}
\item[(1)]
We consider  $\dot W^{s }_p (\R)$ and $\dot B^{s }_{p,q} (\R)$ as the quotient spaces with polynomial space   and so $\dot W^s_p (\R)$ and $\dot B^{s }_{p,q} (\R)$ are normed space.

\item[(2)]
For $ s \in {\mathbb R}$, $1 < p < \infty$, the   homogeneous Sobolev space  $\dot W^{s }_p (\R)$ norm is equivalent to 
\begin{align}\label{0712-5}
  \| {\mathcal F}^{-1} \big( \sum_{-\infty< k < \infty}   |\xi|^s  \hat \Psi(2^{-k} \xi ) \hat{f}\big)\|_{L^p (\R)}.
\end{align}
(See Appendix A in \cite{CJ-jmp})

\item[(3)]
From (2),  we obtain $ \dot W^k_p (\R) = \{ f \in {\mathcal S}'(\R) \, | \, D^\ga f \in  L^p (\R) \quad \mbox{for} \quad |\ga| =k \}$ for $ k \in {\mathbb N} \cup \{ 0 \}$, $ 1 < p < \infty$.

\end{itemize}
\end{rem}

\begin{prop}\label{prop0215} 
\begin{itemize}
\item[(1)]
For  $1 \leq  p, \,\, q_0,\,\, q_1,\,\,  r\leq \infty$ and $ s_1, \, s_2 \in {\mathbb R}$,
\begin{align*}
( \dot W^{s_0 }_p (\R),   \dot W^{s_1 }_p (\R) )_{\te, r} = \dot B^{s }_{p,r} (\R), \qquad
( \dot B^{s_0 }_{p, q_0} (\R),   \dot B^{s_1 }_{p,q_1} (\R) )_{\te, r} = \dot B^{s }_{p,r} (\R),
\end{align*}
where $0 < \te<1$ and  $s = s_0  (1 -\te) + s_1 \te$. For $0< \theta<1$ and $1\leq p\leq \infty$,  we denote by   $(X,Y)_{\theta,p}$  the real interpolation of the normed spaces $X$ and $Y$.

\item[(2)]
Suppose that $ 1 \leq p_0\leq p_1  \leq \infty, \,  1 \leq r_0\leq r_1  \leq \infty$ and $ s_0\geq s_1$ with $s_0 - \frac{n}{p_0} = s_1 - \frac{n}{p_1}$.
Accordingly, the following inclusions hold
\begin{align*}
\dot  W^{s_0 }_{p_0}  (\R) \subset   \dot W^{s_1 }_{p_1} (\R), \quad
\dot  B^{ s_0 }_{p_0, r_0} (\R) \subset   \dot B^{s_1 }_{p_1 ,r_1} (\R).
\end{align*}
 \end{itemize}
\end{prop}
 See Section 6.4 in \cite{BL}.

For the normed space $X$ and interval $(0, T), \, 0< T \leq \infty$, we denote by $L^p(0, T;X), 1\leq p\leq \infty$  the usual Bochner space.
We denote by $L^q_s(0,T;X)$ the weighted Bochner space with the norm
\[
\|f\|_{L^q_s(0,T;X)}=\Big(\int^T_0t^{sq}\|f(t)\|_{X}^q dt\Big)^{\frac{1}{q}}.
\]

\subsection{Useful inequalities}
 
\begin{lemm}
[{\cite[Lemma 2.2]{chae}}]
\label{L21}
Let $s > 0$ and $1 \leq p, \,  q \leq \infty$.
If $1 \leq p \leq r_i$,  $  p_i \leq \infty$ satisfies $\frac1p = \frac1{r_i} + \frac1{p_i}$ for $i=1,2$, then
\[
\norm{fg}_{ \dot B^{s }_{p,q} } 
\leq c \big(
\norm{f}_{ L^{r_1} } 
\norm{g}_{ \dot B^{s }_{p_1,q } } 
+ 
\norm{g}_{ L^{r_2 } } 
\norm{f}_{ \dot B^{s }_{p_2 ,q} } \big).
\]
\end{lemm}

\begin{lemm}
\label{L22}
Let $1 \leq p, \,q, \, r \leq \infty$.
For $ \ga_1< \ga_0 < \ga_2$,
\begin{align*}
\| f\|_{\dot W^{\ga_0}_p (\R) } 
&\leq c  \| f\|^{\frac{\ga_2 -\ga_0}{\ga_2 -\ga_1}}_{\dot W^{\ga_1}_{p} (\R) } \| f\|^{\frac{\ga_0 -\ga_1}{\ga_2 -\ga_1}} _{\dot B^{\ga_2}_{p,q} (\R) }, \qquad 
\| f\|_{\dot B^{\ga_0}_{p,q} (\R) } 
 \leq c \| f\|^{\frac{\ga_2 -\ga_0}{\ga_2 -\ga_1}}_{\dot W^{\ga_1}_{p} (\R) } \| f\|^{\frac{\ga_0 -\ga_1}{\ga_2 -\ga_1}} _{\dot B^{\ga_2}_{p,r} (\R) }.
\end{align*}
\end{lemm}

\begin{proof}
Note that $ \| f \|_{\dot B^{\ga}_{p,\infty} (\R) } \leq c \| f \|_{\dot W^{\ga}_p (\R)} \leq c \| f \|_{\dot B^{\ga}_{p,1} (\R) } $ and $ \| f \|_{\dot B^{\ga}_{p,\infty} (\R) } \leq c \| f \|_{\dot B^{\ga}_{p,q} (\R)} \leq c \| f \|_{\dot B^{\ga}_{p,1} (\R) } $  for $ 1 \leq p, \, q \leq \infty$ and $\ga \in {\mathbb R}$. 
Since $(\dot B^{\ga_1}_{p, \infty} (\R) , \dot B^{\ga_2}_{p,\infty} (\R)   )_{\frac{\ga_0 -\ga_1}{\ga_2 -\ga_1},1  }  =   \dot B^{\ga_0}_{p,1} (\R)$, we get
\[
\| f \|_{\dot W^{\ga_0}_p (\R)} 
\leq c \| f \|_{\dot B^{\ga_0}_{p,1} (\R) } 
\leq c \| f\|^{\frac{\ga_2 -\ga_0}{\ga_2 -\ga_1}}_{\dot B^{\ga_1}_{p, \infty} (\R) } \| f\|^{\frac{\ga_0 -\ga_1}{\ga_2 -\ga_1}} _{\dot B^{\ga_2}_{p,\infty} (\R) }
\leq c \| f\|^{\frac{\ga_2 -\ga_0}{\ga_2 -\ga_1}}_{\dot W^{\ga_1}_{p} (\R) } \| f\|^{\frac{\ga_0 -\ga_1}{\ga_2 -\ga_1}} _{\dot B^{\ga_2}_{p,r} (\R) }.
\]
This is the proof of the first inequality. The proof of the second inequality is similar.
\end{proof}

\begin{lemm}
\label{L23}
Let $0 < s<1$ and $ 1 \leq p, \, q < \infty$. Let $ \rho^*(r) = \sup_{|{\mathbb G}| \leq r} |\rho( {\mathbb G})|$.
If  $ {\mathbb G}, \, {\mathbb H} \in \dot B^s_{p,q}(\R) \cap L^\infty(\R  )$ with $ \| {\mathbb G}\|_{L^\infty(\R  )}, \,\, \| {\mathbb H}\|_{L^\infty(\R  )} \leq \frac12 r_0$ for some $ 0 < r_0$,  then
\begin{align}\label{0114-1}
\begin{split}
\|\rho ({\mathbb G} )\|_{\dot B^{ s}_{p,q} (\R) }  
&\leq c  (D\rho)^* ( r_0)    \|{\mathbb G} \|_{\dot B^s_{p,q} (\R)},\\
 \| \rho ({\mathbb G}) -\rho ({\mathbb H})\|_{\dot B^{s }_{p,q}(\R) }  &  \leq c \Big(  (D\rho)^* ( r_0)   \| {\mathbb G} - {\mathbb H}\|_{\dot B^{s}_{p,q} (\R)} \\
 & \qquad + (D^2\rho)^* (r_0)  \big( \|{\mathbb G}\|_{\dot B^{s}_{p,q} (\R)}
   +     \|{\mathbb H}\|_{\dot B^{s}_{p,q} (\R)}    \big)   \| {\mathbb G}-{\mathbb H}\|_{L^\infty (\R)}   \Big).
 \end{split}
\end{align}
\end{lemm}
We prove Lemma \ref{L23} in Appendix \ref{prooflemmaL23}.

\subsection{Stochastic function spaces}

Let $(\Om, {\mathcal G}, {\mathbb P})$ be a probability space.
Let $\{{\mathcal G}_t : t\geq 0\}$ be a filtration of $\si$-fields ${\mathcal G}_t \subset {\mathcal G}$ with $\mathcal{G}_0$ containing all ${\mathbb P}$-null subsets of $\Om$.
We denote by $B(t)$ a one-dimensional $\{\mathcal G_t \}$-adapted Wiener process defined on $(\Om, {\mathcal G}, {\mathbb P})$.
The expectation of a random variable $f$ is denoted by ${\mathbb E} f= \int_\Omega f(\omega) dP(\omega)$.
Following the standard convention, we omit the argument $\omega$ of random variables $f(\omega)$.

The solution $u$ and data $(u_0,g)$ in \eqref{0410-1} are random variables.
We construct suitable spaces for them using the Besov spaces.
We will use two different kinds of spaces.
The first type emphasizes the regularity in $x$, whereas the second type does the joint regularity in $(t,x)$.
We can consider $u$ and $g$ as Banach space-valued stochastic processes.
Hence, $(\Omega\times(0,\infty),\mathcal{P}, {\mathbb P} \bigotimes\ell(0,\infty))$ is a suitable choice for their common domain, where $ {\mathcal P}$ is the predictable $\si$-field generated by $\{{\mathcal G}_t : t \geq 0\}$ (see, for instance, pp. 84--85 of \cite{MR1661766}) and $\ell(0,\infty)$ is the Lebesgue measure on $(0,\infty)$.

We define ${\mathbb L}^p (\R)$, $\dot {\mathbb W}^k_p (\R)$ and and $\dot {\mathbb B}_{p,q}^s (\R), \,\, s \in {\mathbb R}$ as sets of $L^p (\R)$-valued ${\mathcal F}_0$-measurable random variables and $ \dot B_{p,q}^s (\R)$-valued ${\mathcal F}_0$-measurable random variables with the norms
\begin{align*}
\|f\|_{{\mathbb L}^p (\R)} 
= \big( {\mathbb E} \|f\|^p_{L^p (\R)} \big)^\frac1p, \quad  \|f\|_{\dot {\mathbb W}^k_p (\R)} 
&= \big( {\mathbb E} \|f\|^p_{\dot W^k_p (\R)} \big)^\frac1p, \quad 
\|f\|_{{\mathbb B}_{p,q}^s (\R)} 
= \big( {\mathbb E} \|f\|^q _{B_{p,q}^s (\R)} \big)^\frac1q.
\end{align*}
For stopping time $\tau$, we denote $ (0, \tau|] : =\{ ( \om,t) \in \Om \times (0, \infty) \,\, | \,\, 0 < t \leq \tau (\om) \, \}$. For a Banach space $X$ and $\al \in {\mathbb R}$ we define the stochastic Banach spaces $ {\mathcal L}^r_\al( (0, \tau|], {\mathcal P}; X)$ and ${\mathcal M}^{r_1,r_2}_\al( (0, \tau|], {\mathcal P}; X)$ to be the space of $X$-valued processes with the norms
\begin{align*}
\|f\|_{ {\mathcal L}^r_\al( (0, \tau|], {\mathcal P}; X)}^r
& = {\mathbb E} \int_0^\tau s^{\al r} \|f(s,\cdot)\|_X^rds, \\
\|f\|_{{\mathcal M} ^{r_1,r_2}_\al( (0, \tau|], {\mathcal P}; X)}^{r_2}
& = {\mathbb E} \big(\int_0^\tau s^{\al r_1} \|f(s,\cdot)\|_X^{r_1}ds \big)^{\frac{r_2}{r_1}}.
\end{align*}
Note that ${\mathcal M} ^{r,r}_\al( (0, \tau|], {\mathcal P}; X) = {\mathcal L}^r_\al( (0, \tau|], {\mathcal P}; X)$. 
In particular, if $\tau =\infty$, then we denote
\begin{align*}
 {\mathcal L}^r_\al( (0, \infty|], {\mathcal P}; X) & = {\mathcal L}^r_\al( \Om \times (0, \infty) , {\mathcal P}; X),\\
{\mathcal M}^{r_1,r_2}_\al( (0, \infty|], {\mathcal P}; X)
& = {\mathcal M}^{r_1,r_2}_\al( \Om \times (0, \infty), {\mathcal P}; X).
\end{align*}
Note that the elements of ${\mathcal L}^r_\al( (0, \tau|], {\mathcal P}; X)$ and ${\mathcal M} ^{r_1,r_2}_\al( (0, \tau|], {\mathcal P}; X) $ are treated as functions rather than distributions or classes of equivalent functions.

From now on, we denote $ X(\R)$ for Banach space defined in $\R$ by $ X$. We also denote ${\mathcal L} ^q( (0, \infty |], {\mathcal P}; X (\R)) $ and ${\mathcal M} ^{p,q}_{a} ( (0, \infty |], {\mathcal P}; X (\R)) $ by $ {\mathcal L}^q_a X $ and ${\mathcal M} ^{p,q}_{a} X $ for Banach space $X$, respectively. We also denote $ L^q_a (0, \infty; X(\R))$ by $ L^q_a X $.

\section{Stochastic Stokes equations}
\label{Proof1}
\setcounter{equation}{0}

For simplicity, we assume that $2{\mathbb F} (0) = {\mathbb I}$ is the identity matrix. Let $\si ({\mathbb A}) : = {\mathbb F} ({\mathbb A}) - {\mathbb F}(0)  $. Accordingly, the first equations in \eqref{maineq2} are denoted by
\begin{align*}
du = \Big( \De u - \na p + {\rm div }\,\big( \si(Du) D u - u \otimes u \big) \Big) dt + g dB_t.
\end{align*}

Thus, the following initial  value problem of the Stokes equations in $\R\times (0,\infty)$ is necessary:
\begin{align}\label{maineq-stokes}
\begin{array}{l}\vspace{2mm}
dw = ( \De w - \na \pi + \mbox{div}{\mathcal F} ) dt + g dB_t, \qquad {\rm div} \, w =0 \,\, \mbox{ in } \,\, 
 \R\times (0,\infty),\quad 
 w|_{t=0}= u_0.
\end{array}
\end{align}

\begin{defin}[Weak solution for the stochastic Stokes equations]
\label{D21-2}
Let $u_0$, ${\mathcal F}$ and $g$ satisfy the assumption in Theorem \ref{theorem0410-2}.
We say that $w$ is a weak solution to \eqref{maineq-stokes} if $w \in {\mathcal L}^q_\al L^p $ is a progressively measurable process and
\begin{align*}
&\int_{\R} w(t,x) \cdot \Phi(x) dx
- \int_{\R} u_0(x) \cdot \Phi(x) dx \\
&= \int^t_0 \int_{\R} \Big( w \cdot \Delta \Phi +{\mathcal F} :\nabla \Phi)\Big) dxdt
+ \int_0^t \int_{\R}  g  \Phi dx dB_t
\end{align*}
for all $ 0 < t < \infty $ with probability $1$ for every $\Phi \in C^\infty_0( \R )$ with $\mbox{div}_x\,\Phi=0$.
\end{defin}

 The solution $w$ is decomposed by $ w= w^1 + w^2 + w^3$, where 
\begin{align}
w^1_i (x,t) & = \int_{\R} \Ga (x-y, t) u_{0i} (y) dy,\\
w^2_i (x,t) & = \int_0^t \int_{\R} \big( \delta_{ij}\Ga (x-y, t-s) - R_i R_j \Ga(x-y, t-s) \big) D_{y_j} \, {\mathcal F}_{ij} (y,s) dyds,\\
w^3_i(x,t) & = \int_0^t \int_{\R} \Ga(x -y, t-s) g_i (y,s) dy dB_s.
\end{align}
Here, $\de_{ij}$ is Kronecker Dirac-delta, $ \Ga$ is fundamental solution of heat equation in $\R$ and $R_i$ is $n$-dimensional Riesz transform.

\begin{theo}\label{theo0328-2}

For $ 1 < p< \infty$, $ 1 \leq q \leq \infty$ and $ 0\leq \al, \, a $, 
\begin{align}
\begin{split}
\| w^1\|_{L^q_a L^p}  \leq c \| u_0\|_{\dot B^{-\frac2q -2a}_{p,q}}, & \quad 
\| w^1\|_{L^q_a \dot B^\al_{p, q} }   \leq c \| u_0\|_{\dot B^{\al-\frac2q -2a}_{p,q}},\quad 
\| D w^1\|_{L^\infty L^\infty  }   \leq c \| D u_0\|_{L^\infty }.
\end{split}
\end{align}

\end{theo}
(See \cite{Tr} for the reference)

\begin{theo}\label{theo0328-1}
\begin{itemize}
\item[(1)]
Let $ 1 < p_1 \leq  p < \infty$, $ 1 < q_1 \leq  q< \infty$ and $ 0 \leq  a \leq  a_1$ such that $\frac1q + a +1 = \frac1{q_1} + a_1 +\frac12 +\frac{n}{2p_1} -\frac{n}{2p}$. Then, 
\begin{align}\label{0307-2}
\| w^2\|_{L^q_a L^p} \leq c \| {\mathcal F} \|_{L^{q_1}_{a_1} L^{p_1}}.
\end{align}

\item[(2)]
Let $ 1 < p_1 \leq p < \infty$, $ 1 \leq q_1 \leq q \leq \infty$, $ 0 < \be_1 \leq \be$ and $0 < a \leq a_1 $ such that $-1 -\be +\be_1 < \frac{n}{p_1} -\frac{n}p < 1 -\be+\be_1$ and $ \frac1q +1 +a = \frac1{q_1} +a_1 +\frac12 +\frac{\be}2 -\frac{\be_1}2 +\frac{n}{2p_1} - \frac{n}{2p} $. Then, 
\begin{align}\label{250122-1}
\| w^2\|_{L^q_a \dot B^\be_{p,q}} \leq c \| {\mathcal F} \|_{ L^{q_1}_{a_1}\dot B^{\be_1}_{p_1,q} }.
\end{align}

\item[(3)]
Let  $ 1 < p_1, \, \, q_1 < \infty $,  $ 0 < \be < 2$ and $ 0 \leq a_1$  such that 
$ \frac{\be}2 = \frac{n}{2p_1} +\frac1{q_1} +a_1 $. Then, 
\begin{align}\label{0122-3}
\| D w^2 \|_{ L^\infty  L^\infty} 
& \leq c \| {\mathcal F} \|_{ L^{q_1}_{a_1} \dot W^\be_{p_1 } }.
\end{align}
\end{itemize}
\end{theo}
We proved Theorem \ref{theo0328-1} in Appendix \ref{prooftheorem0328-1}.

\begin{theo}\label{theorem0410-1}
Let $ 1 < p < \infty$, $ 2 \leq q \leq \infty$  and $\al \in {\mathbb R}$. Let  $ 0 \leq a <  \frac12 -\frac1q$ if $ 2 < q$ and $ a =0$ if $ q =2$.
\begin{align*}
\|w^3\|_{{\mathcal L} ^q_a \dot B^\al_{p,q} }
\lesssim \|g\|_{{\mathcal L} ^q_a  \dot B^{\al-1}_{p,q} },\qquad 
\| D w^3\|_{{\mathcal L} ^\infty   L^\infty  } 
 \leq c \Big( \| D_x g \|_{{\mathcal L} ^q_a L^p } + \| D^2_x g \|_{{\mathcal L} ^q_a  L^p } \Big).
\end{align*}
For $1 < p_2 \leq p, \,\, 2 \leq q$ and $ a \leq a_2$ with $\big(\frac{n}{p_2}-\frac{n}{p}\big) + 2(a_2-a) = 1$, 
\begin{align}\label{w3}
\|w^3\|_{{\mathcal L} ^q_a L^p }
\lesssim \|g\|_{{\mathcal L} ^q_{a_2}  L^{p_2} }.
\end{align}

\end{theo}
We prove Theorem \ref{theorem0410-1} in Appendix \ref{proofoflemma0719-1}.

 From Theorem \ref{theo0328-2} to Theorem \ref{theorem0410-1}, we have 
 \begin{theo}\label{theorem0410-2}
 \begin{itemize}
\item[(1)]
Let $ 1 < p_1 \leq  p < \infty$, $ 1 < q_1 \leq q< \infty$ and $ 0 \leq  a \leq  a_1$ such that $\frac1q + a +1 = \frac1{q_1} + a_1 +\frac12 +\frac{n}{2p_1} -\frac{n}{2p}$. Let $1 < p_2 \leq p <\infty, \quad 2 \leq q < \infty$ and $ a \leq a_2$ with $\big(\frac{n}{p_2}-\frac{n}{p}\big) + 2(a_2-a) = 1$. Then, 
\begin{align}
\| w\|_{{\mathcal L} ^q_a  L^p  } \leq c \| u_0\|_{\dot{\mathbb B}^{-\frac2q -2a}_{p,q}  } + \|g\|_{{\mathcal L} ^q_{a_2}   L^{p_2}  } + \| {\mathcal F} \|_{{\mathcal U}^{p_1, q_1}_{a_1}   L^{p_1}   }.
\end{align}

\item[(2)]
Let $ 1 < p_1 \leq p < \infty$, $ 1 \leq q_1 \leq q \leq \infty$, $ 0 < \be_1 \leq \be$ and $0 \leq  a \leq a_1 $ such that $-1 -\be +\be_1 < \frac{n}{p_1} -\frac{n}p < 1 -\be+\be_2$ and $ \frac1q +1 +a = \frac1{q_1} +a_1 +\frac12 +\frac{\be}2 -\frac{\be_1}2 +\frac{n}{2p_1} - \frac{n}{2p} $. Then, 
\begin{align}
\| w\|_{{\mathcal L}^q_a \dot B^\be_{pq}} \leq c \| u_0\|_{\dot {\mathbb B}^{\be-\frac2q -2a}_{p,q}  } + \|g\|_{{\mathcal L} ^q_a \dot B^{\be-1}_{p,q}  } + \| {\mathcal F} \|_{ {\mathcal U}^{p_1, q_1}_{a_1} \dot B^{\be_1}_{p_1,q} }.
\end{align}

\item[(3)]
Let  $  p_1 \leq p, \, \, q_1 \leq q $,  $ 0 < \be < 2$ and $ 0 \leq a_1$  such that 
$ \frac{1}2  \be= \frac{n}{2p_1} +\frac1{q_1} +a_1 $.  Then, 
\begin{align}
\begin{split}
\| D w \|_{ {\mathcal L} ^\infty   L^\infty  } 
&\leq c \| D u_0\|_{{\mathbb L}^\infty  } 
+ \| D g\|_{{\mathcal L} ^q L^{p} }  + \|D^2 g\|_{{\mathcal L} ^q_a  L^{p} } + \| {\mathcal F} \|_{{\mathcal L} ^{q_1}_{a_1} \dot W^\be_{p_1 } }.
\end{split}
\end{align}
\end{itemize}

 \end{theo}

\section{Proof of Theorem \ref{T1}}
\label{nonlinear}
\setcounter{equation}{0}

Define the Banach space ${\mathcal X}$ and its closed subset ${\mathcal T}_R$ as
\begin{align*}
{\mathcal X} := {\mathcal L}^q_a \dot B^{\al}_{p,q} \cap {\mathcal L}^\infty \dot W^1_\infty \cap {\mathcal L}^q_a L^p, \quad 
 {\mathcal X}_R := \{u \in {\mathcal X} : \|u\|_{\mathcal X} \leq R \}.
\end{align*}
Here,  $0 <R \leq r_0$ is to be chosen later, where $ r_0$ is introduced in {\bf Assumption}.
For $s \geq 0$, we define  $ \te: [0, \infty) \ri [0, 1]$ by
\[\theta(s) = \max\{0,\min\{2-s/R,1\} \}.\]
For $v \in {\mathcal X} $, we define the map $\chi_{v(\om)} : [0,\infty) \rightarrow [0,1]$ by
\[
\chi_v := \chi_{v(\om)}(t) = \te (\eta_v(\om, t)  ), 
\]
where $\eta_v(\om, t) : = \|v(\om)\|_{ L^q_\al (0,t;\dot B^{\al}_{p,q} )} + \| D v(\om)\|_{L^\infty (0,t; L^\infty )} + \| v (\om) \|_{L^q_a (0, t; L^p  )}$.   For $\om \in \Om$, let us 
\begin{align}\label{t0}
t_0 (\om) : = \inf_{0< t< \infty} \{   \eta(\om, t) > 2R\}.
\end{align}
Note that for  $ t_0 < t$,    $\eta_v(\om, t) \geq 2R$ and  $\chi_{v(\om)}(t) =0$. Hence, we have 
\begin{align}\label{0529-1}
 \|\chi_v v\| _{L^q_a \dot B^{\al}_{p,q}  } +\|\chi_v D v\|_{L^\infty L^\infty  } +\|\chi_v v\|_{L^q_a L^p }  \leq \eta_v(\om,t_0) \leq 2R.
\end{align}

Let  $ u_0 \in {\mathbb T}, \,\, g \in {\mathcal T}$   with ${\rm div} \, u_0 = 0$ and $ {\rm div} \, g =0$, where ${\mathbb T}$ and ${\mathcal T}$ are defined in \eqref{def1} and \eqref{def2}.  Let 
 \begin{align}
 M_0 : = \| u_0\|_{{\mathbb T}} + \| g\|_{{\mathcal T}}.
 \end{align}

Given $v \in {\mathcal C}_R$, we consider the following  equations
\begin{equation}
\label{E40}
d w_v = \big( \De w_v + \na P_v - {\rm div } \, \Big( \chi_v^2 \big( \si ( D v) D v+ v \otimes v ) \big)\Big)dt + g dB_t,  \quad {\rm div}  \, w_v = 0, \quad w_v|_{t =0} = u_0.
\end{equation}
From the solvability of this problem (see Theorem \ref{theorem0410-2}), we can define the solution map by
\[
S(v)=w_v.
\]
We shall show the existence of a solution by using the Banach fixed point theorem.

\subsection{$S$ is bounded}
\label{subsection0414}

Take $a_1 = 2a, \,  p_1 =\frac{p}2$ and $ q_1 =\frac{q}2$  such that  $\frac1q +1 + a = \frac2{q} + 2a+\frac12 +\frac{n}{p} -\frac{n}{2p}$. Applying (1) of  Theorem \ref{theorem0410-2}, we have 
\begin{align}\label{0326-2}
 \|S(v)\|_{ {\mathcal L}^q_a L^p } & \leq c \big(M_0 + \| \chi_v^2 v \otimes v \|_{{\mathcal M} ^{\frac{q}2,q}_{2a} L^{\frac{p}2 } }
 + \| \chi_v^2\si (D v) D v \|_{{\mathcal M} ^{\frac{q}2,q}_{2a} L^{\frac{p}2 } } \big).
\end{align}

From (2) of Theorem \ref{theorem0410-2},  we have 
\begin{equation}
\label{E33-1}
\begin{split}
\|S(v)\|_{{\mathcal L}^q_a \dot B^{\al}_{p,q} } & \leq c \Big(M_0  + \| \chi_v^2 v \otimes v \|_{{\mathcal M} ^{\frac{q}2,q}_{2a} \dot B^{\al}_{\frac{p}2, q} }  + \| \chi_v^2 \si (Dv) D v \|_{{\mathcal L}^q_a \dot B^{\al-1}_{p,q} } \Big).
\end{split}
\end{equation}

Take $ 0 < \be \leq 1$,  $ p \leq p_1, \,\, q_1 =q$ and $a_1 =2 a$ such that $ \frac12 \be = \frac{n}{2p_1} +\frac1{q} +a $. Applying (3) of  Theorem \ref{theorem0410-2},   we have 
 \begin{align}\label{0326-1}
\| D S (v) \|_{{\mathcal L}^\infty L^\infty } 
& \leq c \big(M_0+ \| \chi_v^2 v \otimes v \|_{{\mathcal L} ^{q}_{a} \dot W^1_p } + \|  \chi_v^2 \si (D v) D v \|_{{\mathcal L} ^{q}_{a} \dot W^\be_{p_1} } \big).
\end{align}

Let us fix $\om \in \Om$.  By Lemma \ref{L21} and   \eqref{0529-1}, we have 
\begin{align}\label{eq1101-1}
\begin{split}
\|\chi_v^2 v \otimes v\|_{L^{\frac{q}2}_{2a} \dot B^{\al}_{\frac{p}2, q} } &= \big( \int_0^\infty \chi_v^q (t) t^{a q} \| v \otimes v \|^{\frac{q}2} _{\dot B^{\al}_{\frac{p}2, q} } dt \big)^\frac{2}q\\
& \leq c\big( \int_0^\infty \chi_v^q (t) t^{a q} \| v\|^{\frac{q}2}_{ L^p} \| v \|^{\frac{q}2}_{ \dot B^{\al}_{p,q} } dt \big)^\frac{2}q\\
& \leq c\| \chi_v v\|_{L^q_a L^p} \|\chi_v v\|_{L^q_a \dot B^\al_{p,q}}\\
& \leq c4R^2.
\end{split}
\end{align}
Similarly,   we have 
\begin{align}\label{0327-3}
\| \chi_v^2 v \otimes v\|_{ L^{\frac{q}2}_{2a} L^{\frac{p}2} } 
 \leq c \|\chi_v v\|_{L^q_a L^p}^{2 } \leq 4cR^2.
\end{align}

Note that $ | \si ({\mathbb G})| \leq c |{\mathbb G}|$ for $ |{\mathbb G}| \leq r_0$ and  $ \| Dv(t)\|_{L^\infty} \leq \eta(\om, t) \leq 2R \leq r_0$ for $ t \leq t_0$,  and so  $  \|\chi_v  \si (Dv)  \|_{L^q_a L^p }  
 \leq c   \| \chi_v D v\|_{ L^q_a L^p } $. From \eqref{0529-1}, we have 
\begin{align}\label{0327-2}
 \| \chi_v^2 \si (D v ) D v\|_{L^{\frac{q}2}_{2a} L^{\frac{p}2 } } 
 \leq c \|\chi_v  \si (Dv)  \|_{L^q_a L^p } \| \chi_v D v\|_{ L^q_a L^p } 
 \leq c   \| \chi_v D v\|^2_{ L^q_a L^p }  \leq cR^2.
\end{align}

Applying  Lemma \ref{L23} for $\rho({\mathbb G}) = \si ({\mathbb G}) {\mathbb G}$, we have 
\begin{align*}
\| \si( D v ) Dv \|^q_{ \dot B^{\al-1 }_{p,q} (\R) }  & \leq c       (D\rho)^* (2\| Dv\|_{L^\infty})  \| D v    \|_{\dot B^{\al -1  }_{p,q} }   \leq c       (D\rho)^* (2\| Dv\|_{L^\infty})   \|  v    \|_{\dot B^{\al   }_{p,q} }.
\end{align*}
From assumption ${\mathbb F}$ and $\si$, we have  $ (D\rho)^* (r ) \leq cr$ for $r \leq r_0$. Then,  $\|\chi_v   (D \rho)^* (2\| Dv\|_{L^\infty})  \|_{L^\infty(0, \infty) } \leq c \| \chi_v Dv\|_{L^\infty L^\infty} \leq cR$.  
From \eqref{0529-1}, we have 
\begin{align}\label{0327-5}
\begin{split}
& \| \chi_v^2 \si (Dv) D v\|_{ L^q_{a} \dot B^{\al-1 }_{p,q}} 
\leq c  \|\chi_v   (D\rho)^* (2\| Dv\|_{L^\infty})  \|_{L^\infty(0, \infty) }   \|  \chi_v  v    \|_{L^q_a\dot B^{\al   }_{p,q} } \leq c R^2.
\end{split}
\end{align}

From \eqref{0529-1},  we have 
\begin{align}\label{0115-1}
\| \chi_v^2 v \otimes v \|_{L^{q}_{a} \dot W^1_p } \leq c \|\chi_v v \|_{L^q_a L^p} \| \chi_v D v\|_{L^\infty L^\infty} \leq c R^2.
\end{align}

Note that   $   \be = \frac{n}{p_1} +\frac2{q} +2a $  such that    $  \be +\frac{n}p -\frac{n}{p_1} = \frac{n}{p_1} +\frac2{q} +2a  +\frac{n}p -\frac{n}{p_1} < \al -1$, from Lemma \ref{L22}, we have 
\begin{align}
 \| \si (D v) D v \|_{ \dot W^\be_{p_1} } \leq c  \| \si (D v) D v \|_{  \dot B^{\be +\frac{n}p -\frac{n}{p_1}}_{p,1} } \leq c\big( \| \si (D v) D v \|_{  L^p}  +  \| \si (D v) D v \|_{  \dot B^{\al-1}_{p,q} }\big).
\end{align}
 Hence, by  \eqref{0327-2}, \eqref{0327-5} and \eqref{0529-1}, we have
\begin{align}\label{0115-2}
\begin{split}
\| \chi_v^2 \si (D v) D v \|_{L^{q}_{a} \dot W^\be_{p_1} } 
& \leq c   R^2. 
\end{split}
\end{align}
From \eqref{0326-2} to \eqref{0115-2} and  from   \eqref{0529-1},  we have 
\begin{align*}
\|S(v)\|_{{\mathcal X}}
&= \|w\|_{{\mathcal X}}
\leq C_1(\|u_0\|_{{\mathbb T}} + \|g\|_{{\mathcal T}}
 + R^2) \leq C_1(\delta + R^2).
\end{align*}
We   take $\de = R^2$   so that
\begin{equation}
\label{E41}
\|S(v)\|_{{\mathcal X}} \le  C_1    R^{2}.
\end{equation}
If $R < \min(\frac1{2 C_1},\frac12 r_0)$, then $S : {\mathcal C}_R \ri {\mathcal C}_R $ is bounded operator.

\subsection{$S$ is contractive}

The following auxiliary result will be needed to prove that $S$ is contractive.

\begin{lemm}\label{lemma0719-2}
For $\om \in \Om$, $t > 0$ and $u, \, v \in {\mathcal C}_R $,
\begin{equation}
\label{E42}
|\chi_u(t) - \chi_v(t)| \leq R^{-1}|\eta_u (t) -\eta_v(t) |.
\end{equation}
\end{lemm}

\begin{proof}
See Appendix \ref{proofoflemma0719-2}.
\end{proof}

We shall show that there is $C_2>0$ such that for $u, \, v \in {\mathcal C}_R$,
\begin{equation}
\label{E43}
\|S(u)-S(v)\|_{{\mathcal X}} \le C_2 R \|u-v\|_{{\mathcal X}}.
\end{equation}
Note that $V = S(u) -S(v)$ and $P = P_u -P_v$ satisfy
\begin{align*}
V_t -\De V + \na P = {\rm div } \, \Big(\chi^2_u u \otimes u + \chi^2_u \si(D u) Du - \chi^2_v v \otimes v - \chi^2_v \si(Dv) Dv \Big), \quad {\rm div }\, V =0
\end{align*}
with $V|_{t=0} =0$.  As the same estimate  with \eqref{0326-2},  \eqref{E33-1}  and \eqref{0326-1}, we have 
\begin{align}\label{0720-1}
\begin{split}
\|S(u)-S(v)\|_{{\mathcal L}^q_\al \dot B^\al_{p,q}}
&\lesssim \|\chi^2_u u \otimes u - \chi^2_v v \otimes v \|_{ {\mathcal M} ^{\frac{q}2,q}_{2a} \dot B^{\al}_{\frac{p}2, q}} + \| \chi^2_u \si(D u) Du - \chi^2_v \si(Dv) Dv\|_{ {\mathcal L}^q_a \dot B^{\al-1}_{p,q} },\\
\|S(u)-S(v)\|_{{\mathcal L}^q_\al L^p }
&\lesssim \|\chi^2_u u \otimes u   - \chi^2_v v \otimes v  \|_{{\mathcal M} ^{q, \frac{q}{2}}_{2\al} L^{\frac{p}{2}} }   + \|  \chi^2_u \si(D u) Du -  \chi^2_v \si(Dv) Dv\|_{{\mathcal M} ^{q, \frac{q}{2}}_{2\al} L^{\frac{p}{2}} },\\
\|S(u)-S(v)\|_{{\mathcal L}^\infty L^\infty }
&\lesssim \|\chi^2_u u \otimes u   - \chi^2_v v \otimes v  \|_{{\mathcal L} ^{q}_{a} \dot W^1_p } + \|  \chi^2_u \si(D u) Du - \chi^2_v \si(Dv) Dv\|_{{\mathcal L} ^{q}_{a} \dot W^\be_{p_1} }.
\end{split}
\end{align}

We denote $\eta(u,v) :=\eta(\om, u, v) = \| u -v\|_{L^q_a L^p} + \| u -v\|_{L^q_a \dot B^\al_{p,q}} + \| D( u -v)\|_{L^\infty L^\infty}$ so that
$|\eta_u(t) - \eta_v (t) | \leq \eta(u,v)$ for all $\om \in \Om$ and $ 0 < t  <  \infty$.

Note that
\begin{align*}
\chi^2_u \si(D u) Du -  \chi^2_v \si(Dv) Dv& =   (\chi_u - \chi_v) \chi_u  \si(D u) Du + \chi_u \chi_v  \big( \si(D u)  Du - \si (Dv)   Dv  \big)  \\
& \quad + (\chi_u -\chi_v ) \chi_v \si(Dv) Dv.
\end{align*}

By direct calculation, we have 
\begin{align*}
\|  \chi_u \chi_v (D\rho)^* ( \| Du\|_{L^\infty} + \| Dv \|_{L^\infty}) \|_{L^\infty (0, \infty) }  &  \leq c \big( \|\chi_u   Du\|_{L^\infty L^\infty} + \|\chi_v  Dv \|_{L^\infty L^\infty}   \big) \leq cR,\\
\|  \sqrt{\chi_u}\sqrt{ \chi_v} (D^2\rho)^* ( \| Du\|_{L^\infty} + \| Dv \|_{L^\infty}) &  \leq c.
\end{align*}
From   Lemma \ref{L23} and \eqref{0529-1}, we have 
\begin{align}
\begin{split}
 & \|\chi_u \chi_v  \big(    \si ( D u )  Du- \si ( Dv) Dv \big)  \|_{L^q_a \dot B^{\al -1}_{p,q}} \\
  &  \leq c \Big(\|  \chi_u \chi_v (D\rho)^* ( \| Du\|_{L^\infty} + \| Dv \|_{L^\infty}) \|_{L^\infty (0, \infty) } \| Du -Dv\|_{L^q_a\dot B^{\al -1}_{p,q}  }\\
 & \quad  + \|  \sqrt{\chi_u}\sqrt{ \chi_v} (D^2\rho )^* ( \| Du\|_{L^\infty} + \| Dv \|_{L^\infty}) \|_{L^\infty (0, \infty) } \big( \| \sqrt{\chi_u} D u\|_{L^q_a\dot B^{\al -1}_{p,q} } + \|\sqrt{\chi_v} D v \|_{L^q_a\dot B^{\al -1}_{p,q} }    \big)  \| D u - Dv\|_{L^\infty L^\infty  }   \Big)\\
 & \leq c R \eta(u,v).
 \end{split}
 \end{align}

From direct calculation, we have 
\begin{align*}
&\|(\chi_u - \chi_v) \chi_v  \si(D v) Dv\|_{L^q_{a} \dot B^{\al-1}_{p,q}}\\
&\leq cR^{-1} |\eta_u (t) -\eta_v(t) | \big(  \|\sqrt{\chi_v}\si ( Dv)\|_{  L^q_a \dot B^{\al -1}_{p,q}} \|  \sqrt{\chi_v}   Dv\|_{L^\infty L^\infty}   
                       +  \| \sqrt{\chi_v}\si ( Dv)\|_{L^\infty L^\infty}\|   \sqrt{\chi_v}   Dv\|_{L^q_a  \dot B^{\al -1}_{p,q}}   \big) \\
& \leq cR \eta(u,v),  
\end{align*}

Hence, we have \begin{align}
\| \chi^2_u \si(D u) Du -  \chi^2_v \si(Dv) Dv \|_{L^q_a \dot B^{\al-1}_{p,q} }
\leq c R \eta(u,v).
\end{align}

Using similar calculation of  \eqref{eq1101-1}, \eqref{0327-3} \eqref{0327-2}, \eqref{0115-1} and \eqref{0115-2},  we have 
 
\begin{align}
\begin{split}
\| \chi^2_u u \otimes u -  \chi^2_v v \otimes v \|_{L^{\frac{q}{2}}_{2\al} L^{\frac{p}{2}} }
\leq c R \eta(u,v),\\
\| \chi^2_u \si(D u) Du -  \chi^2_v \si(Dv) Dv \|_{L^{\frac{q}{2}}_{2\al} L^{\frac{p}{2}} }
\leq c R \eta(u,v)\\
 \|\chi^2_u u \otimes u   - \chi^2_v v \otimes v  \|_{L ^{q}_{a} \dot H^1_p } \leq c R \eta(u,v),\\
\| \chi_u^2 u \otimes   u - \chi_v^2 v \otimes  v  \|_{L^{\frac{q}{2}}_{2\al} \dot B^{\al}_{\frac{p}2,q}}
\leq c R \eta(u,v).
\end{split}
\end{align}

Combining these estimates with \eqref{0720-1}, we obtain  \eqref{E43}.
Moreover, if $R < \frac1{C_2}$, then $S : {\mathcal T}_R \ri {\mathcal T}_R $ is contractive.

\subsection{Existence}

By the Banach fixed point theorem, we have obtained a  solution $u$ to \eqref{E40} in ${\mathcal C}_R$.
Now, we define the stopping time 
\[
\tau(\om) = \inf \{0 \le T \le \infty\, : \eta_u(t)\ge 2R \}.
\]
Then, $(u, \tau)$ is the   solution in the sense of Definition \ref{D21-2}. Notice from definition of $\tau$ that for all $h \in (0,\infty)$, we have 
\[
\{\om \, | \, \tau(\om) \leq h\} = \{ \om \, | \, \eta_u(h)\geq 2R\}.
\]
Using the Chebyshev inequality,  for all $h \in (0,\infty)$, we have 
\[
{\mathbb P} (\tau = 0)
\le {\mathbb P} (\{\tau \leq h\})
\leq \frac1{R} {\mathbb E} \eta_u(h).
\]

By the definition of the stopping time $\tau$ and \eqref{E41},
\begin{align*}
{\mathbb P}(\tau < \infty)
&= {\mathbb E} 1_{\tau < \infty }
\le {\mathbb E} \Bigg(1_{\tau < \infty } \frac{\eta_u (\infty)}{R}\Bigg)  \le \frac{ c \| u\|_{{\mathcal X}} }{R}
\le \frac{5C_1 R^2}{R} = 5C_1 R.
\end{align*}
If $R < \frac{1}{c} \wedge \frac{\ep }{5C_1}$, then
\[
{\mathbb P} (\tau = \infty) = 1 -{\mathbb P} (\tau < \infty)
\geq 1 - \ep.
\]
This completes the proof of the existence in Theorem \ref{T1}.
\qed

\subsection{Uniqueness }
Let $ u_1\in {\mathcal X} $ be another weak solution of of the system \eqref{maineq2} with pressure $p_1$ satisfying $\| u \|_{{\mathcal X} } < c R^2$, where $R$ is found in Subsection \ref{subsection0414}. Let $U = u - u_1$ and $P = p -p_1$. According,
 $(U, P)$ satisfies the equations
\begin{align*}
U_t - \De U + \na P& =-\mbox{div}(u\otimes u+u_1\otimes u_1) -\mbox{div}( \si(D u) Du -\si( D u_1) D u_1),\\
 {\rm div} \, U& =0,
 \mbox{ in }\R\times (0,\infty),\\
 U|_{t=0}= 0.&
\end{align*}

The estimate \eqref{E43} implies that
\begin{align*}
 \| U \|_{ {\mathcal X} }
 \leq c R \| U\|_{ {\mathcal X} }
 < \| U \|_{{\mathcal X} }.
 \end{align*}
This implies that $u \equiv u_1$ in $\R \times (0, \infty)$. Thus, the proof of the uniqueness of Theorem \ref{T1} is successfully proved.

\section*{Funding}

T. C. has been supported by the National Research Foundation of Korea(NRF) grant funded by the Korean government(MSIT) (No. RS-2023-00244630).
M. Y. has been supported by the National Research Foundation of Korea(NRF) grant funded by the Korean government(MSIT) (No. 2021R1A2C4002840).

\appendix

\section{Proof of Lemma  \ref{L23} }
\label{prooflemmaL23}
\setcounter{equation}{0}

Let $\| {\mathbb G}|_{L^\infty}, \,\, \| {\mathbb H}\|_{L^\infty} \leq \frac12 r_0$.  Using \eqref{0307-1} and mean-value Theorem, we have 
\begin{align}\label{0114-1-1}
\begin{split}
\|\rho ({\mathbb G} )\|_{\dot B^{ s}_{p,q} }  & 
\approx 
\left( \int_{\Rn} \frac1{|y|^{n + qs}} \left( \int_{\Rn} |\rho({\mathbb G}  (y + x)) -\rho({\mathbb G} (x)) |^p  dx \right)^{\frac{q}p} dy \right)^\frac1q\\
 & 
\leq c    \| (D \rho)({\mathbb G}  ) \|_{L^\infty} 
\left( \int_{\Rn} \frac1{|y|^{n + qs}} \left( \int_{\Rn} |{\mathbb G}  (y + x)- {\mathbb G} (x)   |^p  dx \right)^{\frac{q}p} dy \right)^\frac1q\\
 & 
\leq c   (D\rho)^* ( r_0 )   \| {\mathbb G}\|_{\dot B^s_{p,q} (\R)}.
\end{split}
\end{align}
We complete the proof of $\eqref{0114-1}_1$.

Next, we prove  $\eqref{0114-1}_2$.   By direct caculation, we have 
\begin{align*}
\begin{split}
&\rho (  {\mathbb G}(x+y )) -\rho(  {\mathbb H}(x +y ) ) - \Big( \rho(  {\mathbb G}(x )) - \rho (  {\mathbb H}(x ) )  \Big)\\
& = \int_0^1 (D\rho)  (\te  {\mathbb G}(x+y ) + (1 -\te)  {\mathbb H}(x +y ) ): \Big( {\mathbb G}(x+y ) -  {\mathbb H}(x+y ) \Big) d\te\\
& \quad - \int_0^1  (D\rho ) (\te  {\mathbb G}(x ) + (1 -\te)  {\mathbb H}(x ) ): \Big( {\mathbb G}(x ) -  {\mathbb H}(x ) \Big) d\te\\
& = \int_0^1  \Big( ( D\rho) (\te  {\mathbb G}(x+y ) + (1 -\te)  {\mathbb H}(x +y ) ) - (D\rho) (\te  {\mathbb G}(x ) + (1 -\te)  {\mathbb H}(x ) )   \Big)\\
& \quad : \Big( {\mathbb G}(x+y ) -  {\mathbb H}(x+y ) \Big) d\te\\
&\quad  +  \int_0^1 (D\rho )(\te  {\mathbb G}(x ) + (1 -\te)  {\mathbb H}(x ) ): \Big({\mathbb G}(x+y ) -  {\mathbb H}(x+y ) - {\mathbb G}(x ) +  {\mathbb H}(x ) \Big) d\te\\
& = I_1 + I_2.
\end{split}
\end{align*}
Here, 
\begin{align*}
|I_2| \leq  c(D\rho)^* (r_0)    \Big| {\mathbb G}(x+y ) -  {\mathbb H}(x+y ) - {\mathbb G}(x ) +  {\mathbb H}(x )  \Big|.  
\end{align*}

For $I_1$, 
\begin{align*}
&  (D \rho) \big(\te  {\mathbb G}(x+y ) + (1 -\te)  {\mathbb H}(x +y ) \big) - (D\rho)\big(\te  {\mathbb G}(x ) + (1 -\te)  {\mathbb H}(x ) \big)   \\
& = \int_0^1 \frac{d}{d \ga}  (D \rho) \Big(  \ga \big( \te  {\mathbb G}(x+y ) + (1 -\te)  {\mathbb H}(x +y )    \big) + (1 -\ga) \big(  \te  {\mathbb G}(x ) + (1 -\te)  {\mathbb H}(x ) \big)     \Big) d\ga\\
& = \int_0^1    ( D^2 \rho ) \Big(  \ga \big( \te  {\mathbb G}(x+y ) + (1 -\te)  {\mathbb H}(x +y )    \big) + (1 -\ga) \big(  \te  {\mathbb G}(x ) + (1 -\te)  {\mathbb H}(x ) \big)     \Big)\\
& :: \Big(\te\big(  {\mathbb G}(x+y )  -   {\mathbb G}(x )  \big) + (1 -\te) \big(   {\mathbb H}(x+y )  -   {\mathbb H}(x )    \big)   \Big) d\ga.
\end{align*}
Hence,  we have 
\begin{align*}
|I_1| & \leq c  (D^2\rho)^* (r_0)   
\Big( \big|  {\mathbb G}(x+y )  -   {\mathbb G}(x ) \big| + \big|     {\mathbb H}(x+y )  -   {\mathbb H}(x )  \big|  \Big) \big|   {\mathbb G}(x+y ) -   {\mathbb H}(x+y  ) \big|   \Big).
\end{align*}
Summing all  estimate, we have 
\begin{align*}
 \| \rho (  {\mathbb G}) -\rho (   {\mathbb H})\|_{\dot B^{s }_{p,q} (\R ) }& \leq c    \Big(   \int_{{\mathbb R}^{n}}  
\frac1{|y|^{n+qs  } } \\
& \quad  \times  \big( \int_{{\mathbb R}^{n }}  |\rho ({\mathbb G}(x+y) )- \rho( {\mathbb H}(x +y ))  - \rho ({\mathbb G}(x) ) + \rho  ({\mathbb H} (x ))  |^p  dx \big)^{\frac{q}p}  dy  \Big)^\frac1q\\
& \leq  c \Big(  (D\rho)^* (r_0)   \| {\mathbb G} - {\mathbb H}\|_{\dot B^{s}_{p,q} (\R)}\\
& \qquad  + (D^2\rho)^* ( r_0 )   \big( \|{\mathbb G}\|_{\dot B^{s}_{p,q} (\R)}   +    \|{\mathbb H}\|_{\dot B^{s}_{p,q} (\R)}    \big)    \| {\mathbb G}-{\mathbb H}\|_{L^\infty (\R)}  \Big).
\end{align*}
 We complete the proof of Lemma  \ref{L23}.

\section{Proof of Theorem \ref{theo0328-1}}
\label{prooftheorem0328-1}
\setcounter{equation}{0}

Note that for $ 1 < p_1 < p < \infty$,  from Young's inequality, we have 
\begin{align*}
\| w^2(t)\|_{L^p} & \leq \int_0^t (t -s)^{-\frac12 -\frac{n}{2p_1} +\frac{n}{2p}} \| {\mathcal F} (s) \|_{L^{p_1}} ds.
\end{align*}
From Theorem 1.2 in \cite{DM}, for $\frac1q + a +1 = \frac1{q_1} + a_1 +\frac12 +\frac{n}{2p_1} -\frac{n}{2p}$, we have 
\begin{align}
\| w^2\|_{L^q_a L^p} \leq c \|{\mathcal F}\|_{L^{q_1}_{a_1} L^{p_1}}.
\end{align}
We complete the proof of  \eqref{0307-2}.

 Let $ 0 \leq k_1 < \be_1  <  k_2  < \be <  k_3$ for some $ k_i \in {\mathbb N} \cup \{0\}$. 
Note that for $ 1 \leq p_1 \leq p \leq \infty$ and $ 0 \leq k_1 \leq k_2, \,\, k_i \in {\mathbb N} \cup \{0\}$,
\begin{align}
\begin{split}
\| D^{k_3}_x  \na\Ga_{t} *  f\|_{L^p} \leq c  t^{-\frac12  -\frac{k_3}2 +\frac{k_2}2 +\frac{n}{2p} -\frac{n}{2p_1}}  \| D^{k_2}_xf \|_{L^{p_1}},\\
\| D^{k_2}_x \na\Ga_{t} *  f\|_{L^p} \leq c  t^{-\frac12  -\frac{k_2}2 +\frac{k_1}2 +\frac{n}{2p} -\frac{n}{2p_1}}  \| D^{k_1}_xf \|_{L^{p_1}}.
\end{split}
\end{align}
Using the real interpolation and complex interpolation (see (1) of Proposition \ref{prop0215}), we have 
\begin{align}\label{250122-2}
\begin{split}
  \| \na \Ga_{t} * f\|_{\dot B^\be_{p,q} } \leq     c  t^{-\frac12 -\frac{\be}2 +\frac{\be_1}2 -\frac{n}{2p_1} + \frac{n}{2p} } \| f\|_{\dot B^{\be_1}_{p_1,q} },\quad  \| \na \Ga_{t} * f\|_{\dot W^\be_{p} }  \leq    c  t^{-\frac12 -\frac{\be}2 +\frac{\be_1}2 -\frac{n}{2p_1} + \frac{n}{2p} } \| f\|_{\dot H^{\be_1}_{p_1} }. 
 \end{split}
\end{align}

Then, for $ 1 <  p_1 \leq p < \infty$, $ 1  \leq  q \leq \infty$ and $ 0 < \be_1 \leq \be$, we have 
\begin{align}
\begin{split}
\| w^2 (t) \|_{\dot B^{\be}_{p,q} } & \leq \int_0^t \| \na \Ga_{t-s} *{\mathcal F} (s)\|_{\dot B^\be_{p,q} } ds \leq \int_0^t (t-s)^{-\frac12 -\frac{\be}2 +\frac{\be_1}2 -\frac{n}{2p_1} + \frac{n}{2p} } \| {\mathcal F}(s)\|_{\dot B^{\be_1}_{p_1,q} } ds. 
 \end{split}
\end{align}
Hence, for $ \frac1q +1 +a = \frac1{q_1} + a_1 +\frac12 +\frac{\be}2 -\frac{\be_1}2 +\frac{n}{2p_1} - \frac{n}{2p} $, $ q_1 \leq q, \,\, p_1 \leq p, \,\, 0 \leq a_1 \leq a$, we have 
\begin{align}
\| w^2\|_{L^q_a  \dot B^\be_{p,q} } \leq c \| {\mathcal F}\|_{ L^{q_1}_{a_1} \dot B^{\be_1}_{p_1,q}  }.
\end{align}
This implies \eqref{250122-1}.

Note that  for $ \leq \be \leq 2$, 
\begin{align}
 \widehat{D_{x_i} \na \Ga_t * f  }(\xi) =  \frac{\xi}{|\xi|} \frac{ \xi_k}{|\xi|}  |\xi|^{2-\be}  e^{-t |\xi|^2}  |\xi|^\be  \hat f(\xi). 
\end{align}
Hence, we have 
\begin{align}
D_{x_i} \na \Ga_t * f  (x) =     R_i I_{ 2 -\be} \Ga_t  *  R I_\be  f (x), 
\end{align}
where $R = (R_1, \cdots, R_n)$ is Riesz tranform and $ I_\be g$ is defined in \eqref{0219-1}.

Let $ 1 < p_1, \,\, q_1 < \infty$, $ 0 \leq a_1$ and $ 0 <\be$ such that  $ \frac{1}2  \be= \frac{n}{2p_1} +\frac1{q_1} +a_1 $. From $\eqref{250122-2}_2$, we have  
\begin{align}
\begin{split}
  \|D_{x_i} w^2  (t)\|_{L^\infty} &\leq   c_n \int_0^t \|   R_i I_{ 2 -\be} \Ga_{t-s} *   R_k I_\be {\mathcal F}(s)\|_{L^\infty} ds\\
& \leq \int_0^t \| \Ga ( t-s)\|_{\dot W^{2 -\be}_{p_1'}} \| {\mathcal F}(s)\|_{\dot W^\be_{p_1}}ds\\
& \leq c \int_0^t ( t-s)^{-1 +\frac12\be - \frac{n}{2p_1}} \| {\mathcal F}(s)\|_{\dot W^\be_{p_1 }} ds\\
& \leq c \| {\mathcal F}\|_{L^{q_1}_{a_1}(0, \infty; \dot H^\be_{p_1 } )} \Big(\int_0^t ( t-s)^{ (-1 +\frac12\be - \frac{n}{2p_1}) q'_1} s^{-a_1 q'_1} ds \Big)^{\frac1{q'_1}}\\
& \leq c \|{\mathcal F}\|_{L^{q_1}_{a_1}  \dot W^\be_{p_1 } }.
\end{split}
\end{align}

Hence, for $ \frac{1}2 \be  = \frac{n}{2p_1} +\frac1{q_1} +a_1 $, we have 
\begin{align}\label{240124-1}
\| D w^2 \|_{ L^\infty L^\infty} 
& \leq c \|{\mathcal F}\|_{ L^{q_1}_{a_1}   \dot W^\be_{p_1} }.
\end{align}
This implies \eqref{0122-3}.

\section{Proof of Theorem \ref{theorem0410-1} }
\label{proofoflemma0719-1}
\setcounter{equation}{0}

See Lemma 4.3 in \cite{CY} for the proof of \eqref{w3}.

From the definition of Besov space, we have 
\begin{align*}
{\mathbb E} \| w_3\|_{L^q_a  ; \dot B^\al_{p, q} }^q = \sum 2^{q\al k} {\mathbb E} \int_0^\infty t^{a q}\|w_3 * \phi_k\|_{L^p}^q dt. 
\end{align*}

We now recall the following estimate 
\begin{align}\label{0316-1}
 \big(\int_{\R} |  \Ga_{t-s} * g* \phi_k(x,s) dx|^p \big)^\frac1p \leq ce^{-c_0(t-s)2^{2k}} \|g* \phi_k(s)\|_{L^{p}}.
\end{align}
 
From Theorem 1.1 in \cite{JML}  and Minkowski's integral inequality, we have
\begin{align*}
{\mathbb E} \int_0^\infty t^{a q} \|w_3 * \phi_k(t)\|_{L^p}^q dt
&\leq c {\mathbb E} \int_0^\infty t^{a q} \left(\int_{\R} \int_0^t |\Ga_{t-s} * g* \phi_k(x,s)\big)^2 ds|^\frac{p}2 dx \right)^{\frac{q}p} dt \\
&\leq c{\mathbb E} \int_0^\infty t^{a q} \left(\int_0^t \big(\int_{\R} |\Ga_{t-s} * g* \phi_k(x,s)|^pdx\big)^\frac2p ds\right)^\frac{q}2 dt\\
&\leq c{\mathbb E}\int_0^\infty t^{a q}\left(\int_0^t e^{- c_0 (t-s)2^{2k}} \| g * \phi_k(s)\|^2_{L^{p}(\R)} ds\right)^\frac{q}2 dt.
\end{align*}

If $ q =2$ and $ a =0$, then
\begin{align*}
{\mathbb E} \int_0^\infty   \|w_3 * \phi_k(t)\|_{L^p}^2 dt
&\leq c{\mathbb E}\int_0^\infty  \| g * \phi_k(s)\|^2_{L^{p}(\R)}  \int_s^\infty    e^{- c_0 (t-s)2^{2k}} dtds \\
&\leq c{\mathbb E}\int_0^\infty 2^{-qk} \| g * \phi_k(s)\|^2_{L^{p}(\R)} ds. 
\end{align*}

Let $ q > 2$ and $ 0 \leq a < \frac12 -\frac1q$. 
We fix $k \in {\mathbb Z}$.   
\begin{align*}
\int_0^t s^{-2a \frac{q}{q-2}} e^{-c_0\frac{q}{q-2}(t-s)2^{2k}} ds  & \leq c \Big(  t^{-2a \frac{q}{q-2}}\int_{\frac{t}2}^t  e^{-c_0\frac{q}{q-2}(t-s)2^{2k}} ds + e^{-\frac{c_0}2 \frac{q}{q-2}t2^{2k}} \int_0^{\frac{t}2} s^{-2a \frac{q}{q-2}}  ds \Big)\\
& \leq c \Big(  t^{-2a \frac{q}{q-2}} \min(t, 2^{-2k})  + e^{-\frac{c_0}2 \frac{q}{q-2}t2^{2k}}  t^{-2a \frac{q}{q-2} +1 }   \Big)\\
& \leq c  t^{-2a \frac{q}{q-2}} \min(t, 2^{-2k}).
\end{align*}
Hence, we have 
\begin{align*}
{\mathbb E} \int_0^\infty t^{a q} \|w_3 * \phi_k(t)\|_{L^p}^q dt
 & \leq c \int_0^\infty t^{a q} \big( \int_0^{t  }  s^{-2a \frac{q}{q-2}} e^{-c_0 \frac{q}{q-2}(t-s)2^{2k}} ds\big)^{\frac{q}2 ( 1 -\frac2{q} ) } \\
 & \qquad 
 \times  \int_0^t s^{a q} e^{-c_0 q(t-s)2^{2k}} \| g * \phi_k(s)\|^q_{L^{p}(\R)} ds dt\\
 & \leq c \int_0^\infty  \min(t, 2^{-2k})^{ \frac{q}2 -1}
 \int_0^t s^{a q } e^{-c_0 \frac{q}{q-2}(t-s)2^{2k}} \| g * \phi_k(s)\|^q_{L^{p}(\R)} ds dt\\
& \leq c \int_0^\infty s^{a q}  \| g * \phi_k(s)\|^q_{L^{p}(\R)} \int_s^\infty   \min(t, 2^{-2k})^{ \frac{q}2 -1} e^{-c_0\frac{q}{q-2}(t-s)2^{2k}}   dtds \\
& \leq c \int_0^\infty s^{a q}  \| g * \phi_k(s)\|^q_{L^{p}(\R)} \int_s^\infty   2^{-(q -2)k}   e^{-c_0\frac{q}{q-2}(t-s)2^{2k}}   dtds \\
& \leq c \int_0^\infty s^{aq} 2^{- q k } \| g * \phi_k(s)\|^q_{L^{p}(\R)} ds.
\end{align*}

Combining all estimate, we obtain that
\begin{align*}
\|w_3\|_{{\mathcal L} ^q_a(( 0, \infty |],{\mathcal P};\dot B^\al_{pq}(\R))}
\lesssim \|g\|_{{\mathcal L}_{a}^{q}(( 0, \infty |], {\mathcal P}; \dot B^{\al-1}_{pq} (\R))}.
\end{align*}
We complete the proof of the first quantity of Theorem \ref{theorem0410-1}.

Since $\dot B^{\frac{n}{p} }_{p,1} = ( L^{p}, \dot W^1_{p})_{p, 1 - \frac{n}{p}}$ for $ n< p $, we have 
\begin{align*}
\| D w^3 (t)\|_{L^\infty} \leq c \| D w^3(t) \|_{\dot B^{\frac{n}{p} }_{p,1}} \leq c \|D w^3(t) \|_{L^{p} }^{ 1 -\frac{n}{p}} \| D w^3(t) \|_{\dot W^1_{ p }}^{\frac{n}{p}} \leq c \big(\| D w^3(t) \|_{L^{p} } + \| D w^3(t) \|_{\dot W^1_{p} } \big).
\end{align*}
Hence, we have 
\begin{align}\label{0412-1}
{\mathbb E}\| D w^3\|_{L^\infty  } & \leq \Big( {\mathbb E} \sup_{ 0 < t < \infty} \big(\|D w^3(t) \|^q_{L^{p} } + \| D w^3(t) \|^q_{\dot W^1_{p} } \big) \Big)^\frac1q.
\end{align}
From Theore 1.1 in  \cite{VWK}, for $ 2 \leq q < \infty$, we have 
\begin{align} 
\begin{split}
{\mathbb E} \sup_{ 0 < t < \infty} \| D w^3(t) \|^q_{L^{p} } 
 & \leq c   {\mathbb E} \int_0^\infty  \big( \int_{\R}| \int_0^t |\Ga_{t-s}* D g(s,x)|^2 ds|^{\frac{p}2} dx \big)^{\frac{q}p} dt.
\end{split}
 \end{align}
 Applying the proof of   Lemma 4.3 in \cite{CY}, we obtain  
 \begin{align}\label{0412-2}
\begin{split}
{\mathbb E} \sup_{ 0 < t < \infty} \| D w^3(t) \|^q_{L^{p} } 
 &   \leq c {\mathbb E} \int_0^\infty \|  D g(s)  \|_{L^p}^q ds.
\end{split}
 \end{align}
 
 Similarly, we have 
 \begin{align}\label{0412-3}
 {\mathbb E} \sup_{ 0 < t < \infty} \| D w^3(t) \|^q_{\dot H^1_{p} } & 
 \leq c {\mathbb E} \int_0^\infty \| D^2_x g(s) \|_{L^p}^{p} ds.
\end{align}
From \eqref{0412-1} to \eqref{0412-3}, we have 
\begin{align}
{\mathbb E} \| D w^3\|_{L^\infty L^\infty  } & \leq c \Big( {\mathbb E} \int_0^\infty \Big( \|   D_x g(s) \|_{L^p}^q + \| D^2_x g(s)  \|_{L^p}^q \Big) ds \Big)^\frac1q.
\end{align}
We complete the proof of the second quantity of Theorem \ref{theorem0410-1}.

\section{Proof of Lemma \ref{lemma0719-2}}
\label{proofoflemma0719-2}
\setcounter{equation}{0}

Let $ \eta_v( t) : = \|v(\om)\|_{L^q_\al (0,t; \dot B^\al_{p,q} (\R))} + \|Dv(\om)\|_{L^\infty_\al (0,t; L^\infty (\R))}+ \|v(\om)\|_{L^q_\al (0,t; L^p (\R))}.$

Note first that $\chi_{v(\om)} (t) = 1$ if $ \eta_v (t) \leq R$ and that $\chi_{v(\om)} (t) = 0$ if $ \eta_v(t)\ge 2R$.
To prove the claim, we may consider the following six cases.
\begin{enumerate}
\item
If $\eta_u(t) \leq \eta_v (t) \le R$, then $|\chi_u - \chi_v| = |1-1| = 0$.
\item
If $\eta_u(t) \leq R \le \eta_v (t) \le 2R$, then
\begin{align*}
|\chi_u - \chi_v|
&= |1-(2-R^{-1}\eta_v(t))| \\
&= R^{-1}|\eta_v(t)-R| \\
&\leq R^{-1}|\eta_v(t)-\eta_u(t) |.
\end{align*}
\item
If $\eta_u(t) \leq R \leq 2R \le \eta_v(t) $, then
\begin{align*}
|\chi_u - \chi_v|
&= |1-0| \le \frac{\eta_v(t)-\eta_u(t)}{2R-R} \\
&\leq R^{-1} (\eta_v(t)-\eta_u(t) ).
\end{align*}
\item
If $R \leq \eta_u(t) \le \eta_v(t) \leq 2R$, then
\begin{align*}
|\chi_u - \chi_v|
&= |(2-R^{-1}\eta_u(t))-(2-R^{-1}\eta_v(t))| \\
&= R^{-1}|\eta_v(t)-\eta_v(t)|.
\end{align*}
\item
If $R \le \eta_u(t) \leq 2R \leq \eta_v(t)$, then
\begin{align*}
|\chi_u - \chi_v|
&= |(2-R^{-1}\eta_u(t))-0| \\
&= R^{-1}|2R-\eta_u(t)| \\
&\le R^{-1}|\eta_v(t)-\eta_u(t)|.
\end{align*}
\item
If $2R \le \eta_u(t) \le \eta_v(t)$, then $|\chi_u - \chi_v| = |0-0| = 0$.
\end{enumerate}
Thus, \eqref{E42} always holds.

\end{document}